\documentclass[runningheads]{llncs}

\usepackage{amsmath}
\usepackage{amssymb}
\usepackage{mathrsfs}
\usepackage{graphicx}
\usepackage{tikz}
\usetikzlibrary{decorations.markings}
\usetikzlibrary{decorations.pathmorphing}
\usetikzlibrary{arrows, snakes}
\usetikzlibrary{shapes.geometric, calc,decorations.pathreplacing}
\usetikzlibrary{backgrounds}
\usetikzlibrary{fit}
\usetikzlibrary{decorations.pathreplacing}
\input xy
\xyoption{all}

\tikzset{middlearrow/.style={
		decoration={markings,
			mark= at position 0.5 with {\arrow{#1}} ,
		},
		postaction={decorate}
	}
}

\newcommand{\proj}{\mathrm{proj}}

\DeclareMathOperator*{\argmin}{arg\,min}

\begin{document}
\title{Projections with logarithmic divergences\thanks{This research is supported by NSERC Discovery Grant RGPIN-2019-04419 and a Connaught New Researcher Award.}}
%
%
\author{Zhixu Tao \and
Ting-Kam Leonard Wong}
\authorrunning{Z.~Tao and T.-K.~L.~Wong}
%
\institute{University of Toronto\\
\email{zhixu.tao@mail.utoronto.ca}\\
\email{tkl.wong@utoronto.ca}}
\maketitle              
\begin{abstract}
In information geometry, generalized exponential families and statistical manifolds with curvature are under active investigation in recent years. In this paper we consider the statistical manifold induced by a logarithmic $L^{(\alpha)}$-divergence which generalizes the Bregman divergence. It is known that such a manifold is dually projectively flat with constant negative sectional curvature, and is closely related to the $\mathcal{F}^{(\alpha)}$-family, a generalized exponential family introduced by the second author \cite{W18}. Our main result constructs a dual foliation of the statistical manifold, i.e., an orthogonal decomposition consisting of primal and dual autoparallel submanifolds. This decomposition, which can be naturally interpreted in terms of primal and dual projections with respect to the logarithmic divergence, extends the dual foliation of a dually flat manifold studied by Amari \cite{A01}. As an application, we formulate a new $L^{(\alpha)}$-PCA problem which generalizes the exponential family PCA \cite{CDS02}.

\end{abstract}
\section{Introduction} \label{sec:intro}
A cornerstone of information geometry \cite{A16,AJVS17} is the dually flat geometry induced by a Bregman divergence \cite{NA82}. This geometry underlies the exponential and mixture families and explains their efficacy in statistical applications. A natural direction is to study extensions of the dually flat geometry and exponential/mixture families as well as their applications. Motivated by optimal transport, in a series of papers \cite{PW16,PW18,W18,W19} Pal and the second author developed the $L^{(\alpha)}$-divergence which is a logarithmic extension of the Bregman divergence. Let $\alpha > 0$ be a constant. Given a convex domain $\Theta \subset \mathbb{R}^d$ and a differentiable $\alpha$-exponentially concave function $\varphi: \Theta \rightarrow \mathbb{R}$ (i.e., $\Phi = e^{\alpha \varphi}$ is concave), we define the $L^{(\alpha)}$-divergence ${\bf L}_{\varphi}^{(\alpha)} : \Theta \times \Theta \rightarrow [0, \infty)$ by
\begin{equation} \label{eqn:L.alpha.div}
{\bf L}_{\varphi}^{(\alpha)}[\theta : \theta'] = \frac{1}{\alpha} \log (1 + \alpha \mathsf{D} \varphi(\theta') \cdot (\theta - \theta')) - (\varphi(\theta) - \varphi(\theta')),
\end{equation}
where $\mathsf{D}\varphi$ is the Euclidean gradient and $\cdot$ is the Euclidean dot product. We recover the Bregman divergence (of a concave function) by letting $\alpha \rightarrow 0^+$.  Throughout this paper we work under the regularity conditions stated in \cite[Condition 7]{W18}; in particular, $\varphi$ is smooth and the Hessian $\mathsf{D}^2 \Phi$ is strictly negative definite (so $\Phi$ is strictly concave). By the general theory of Eguchi \cite{E83}, the divergence ${\bf L}_{\varphi}^{(\alpha)}$ induces a dualistic structure $(g, \nabla, \nabla^*)$ consisting of a Riemannian metric $g$ and a pair $(\nabla, \nabla^*)$ of torsion-free affine connections that are dual with respect to $g$. It was shown in \cite{PW18,W18} that the induced geometry is dually projectively flat with constant negative sectional curvature $-\alpha$ (see Section \ref{sec:prelim} for a brief review). Moreover, the $L^{(\alpha)}$-divergence is a canonical divergence (in the sense of \cite{AA15}) for such a geometry. Thus, this geometry can be regarded as the constant-curvature analogue of the dually flat manifold. The geometric meaning of the curvature $-\alpha$ was investigated further in \cite{WY19,WY19b}. Also see \cite{PW18,PW18b,WY19b} for connections with the theory of optimal transport. In \cite{W18} we also introduced the $\mathcal{F}^{(\alpha)}$-family, a parameterized density of the form
\begin{equation} \label{eqn:F.alpha.family}
p(x; \theta) = (1 + \alpha \theta \cdot F(x))^{-1/\alpha} e^{\varphi(\theta)},
\end{equation}
and showed that it is naturally compatible with the $L^{(\alpha)}$-divergence. To wit, the potential function $\varphi$ in \eqref{eqn:F.alpha.family} can be shown to be $\alpha$-exponentially concave, and its $L^{(\alpha)}$-divergence is the R\'{e}nyi divergence of order $1 + \alpha$. Note that letting $\alpha \rightarrow 0^+$ in \eqref{eqn:F.alpha.family} recovers the usual exponential family. In a forthcoming paper we will study in detail the relationship between the $\mathcal{F}^{(\alpha)}$-family and the $q$-exponential family \cite{N11}, where $q = 1 + \alpha$.

To prepare for statistical applications of the logarithmic divergence and the $\mathcal{F}^{(\alpha)}$-family, in this paper we study primal and dual projections with respect to the $L^{(\alpha)}$-divergence. These are divergence minimization problems of the form
\begin{equation} \label{eqn:projecion.problems}
\inf_{Q \in \mathcal{A}} {\bf L}_{\varphi}^{(\alpha)} [ P : Q] \quad \text{and} \quad \inf_{Q \in \mathcal{A}} {\bf L}_{\varphi}^{(\alpha)} [ Q : P],
\end{equation}
where $P$ is a given point and $\mathcal{A}$ is a submanifold of the underlying manifold. The optimal solutions in \eqref{eqn:projecion.problems} are respectively the primal and dual projections of $P$ onto $\mathcal{A}$. Our main result, presented in Section \ref{sec:foliation}, provides an orthogonal foliation of the manifold in terms of $\nabla$ and $\nabla^*$-autoparallel submanifolds. This extends the orthogonal foliation of a dually flat manifold constructed by Amari \cite{A01}, and shows that projections with respect to an $L^{(\alpha)}$-divergence are well-behaved. As an application,  we formulate in Section \ref{sec:pca} a nonlinear dimension reduction problem that we call the $L^{(\alpha)}$-PCA. In a nutshell, the $L^{(\alpha)}$-PCA problem extends the exponential family PCA \cite{CDS02} to the $\mathcal{F}^{(\alpha)}$-family \eqref{eqn:F.alpha.family}. In future research we will study further properties of the $L^{(\alpha)}$-PCA problem, including optimization algorithms and statistical applications.

\section{Preliminaries} \label{sec:prelim}
Consider an $L^{(\alpha)}$-divergence ${\bf L}_{\varphi}^{(\alpha)}$ as in \eqref{eqn:L.alpha.div}. We regard the convex set $\Theta$ as the domain of the (global) primal coordinate system $\theta$. We denote the underlying manifold by $\mathcal{S}$. For $P \in \mathcal{S}$, we let $\theta_P \in \Theta$ be its primal coordinates. We define the dual coordinate system by
\begin{equation} \label{eqn:dual.coordinates}
\eta_P = \mathsf{T}(\theta_P) := \frac{\mathsf{D} \varphi(\theta_P)}{1 - \alpha \mathsf{D} \varphi(\theta_P) \cdot \theta_P}.
\end{equation}
We call the mapping $\theta \mapsto \eta = \mathsf{T}(\theta)$ the $\alpha$-Legendre transformation and let $\Omega$ be the range of $\eta$. This transformation corresponds to the $\alpha$-conjugate defined by
\begin{equation*} 
\psi(y) = \inf_{\theta \in \Theta} \left(\frac{1}{\alpha}  \log (1 + \alpha \theta \cdot y) - \varphi(\theta) \right) .
\end{equation*}
It can be shown that $\psi$ is also $\alpha$-exponentially concave. The inverse of $\mathsf{T}$ is the $\alpha$-Legendre transform of $\psi$, and we have the self-dual expression
\begin{equation} \label{eqn:self.dual}
{\bf D}[P : Q] := {\bf L}_{\varphi}^{(\alpha)} [ \theta_P : \theta_Q] = {\bf L}_{\psi}^{(\alpha)} [ \eta_Q : \eta_P].
\end{equation}
We refer the reader to \cite[Section 3]{W18} for more details about this generalized duality. Note that \eqref{eqn:self.dual} defines a divergence ${\bf D}$ on $\mathcal{S}$. 

Let $(g, \nabla, \nabla^*)$ be the duallistic structure, in the sense of \cite{E83}, induced by the divergence ${\bf D}$ on $\mathcal{S}$. The following lemma expresses the Riemannian metric $g$ in terms of the coordinate frames $(\frac{\partial}{\partial \theta_i})$ and $(\frac{\partial}{\partial \eta_j})$ (also see \cite[Remark 6]{W18}).

\begin{lemma} \cite[Proposition 8]{W18} \label{lem:metric}
Let $\langle \cdot, \cdot \rangle$ be the Riemannian inner product. Then
\begin{equation} \label{eqn:metric}
\left\langle \frac{\partial}{\partial \theta_i}, \frac{\partial}{\partial \eta_j} \right\rangle = \frac{-1}{\Pi} \delta_{ij} + \frac{\alpha}{\Pi^2} \theta_j \eta_i, \quad \Pi = 1 + \alpha \theta \cdot \eta.
\end{equation}
\end{lemma}

We call $\nabla$ the primal connection and $\nabla^*$ the dual connection. A primal geodesic is a curve $\gamma$ such that $\nabla_{\dot{\gamma}} \dot{\gamma} = 0$, and a dual geodesic is defined analogously. As shown in \cite[Section 6]{W18}, the geometry is dually projectively flat in the following sense: If $\gamma: [0, 1] \rightarrow \mathcal{S}$ is a primal geodesic, then under the primal coordinate system,  $\theta_{\gamma(t)}$ is a straight line up to a time reparameterization. Similarly, a dual geodesic is a time-changed straight line under the dual coordinate system. We also have the following generalized Pythagorean theorem.

\begin{theorem} [Generalized Pythagorean theorem] \cite[Theorem 16]{W18}
For $P, Q, R \in \mathcal{S}$, the generalized Pythagorean relation 
\begin{equation*} 
{\bf D}[Q : P] + {\bf D}[R : Q] = {\bf D}[R : P]
\end{equation*}
holds if and only if the primal geodesic from $Q$ to $R$ meets $g$-orthogonally to the dual geodesic from $Q$ to $P$.
\end{theorem}

\section{Dual foliation and projection} \label{sec:foliation}
In this section we construct an orthogonal decomposition of $\mathcal{S}$. We begin with some notations. Let $\mathcal{A} \subset \mathcal{S}$ be a submanifold of $\mathcal{S}$. Thanks to the dual projective flatness, we say that $\mathcal{A}$ is $\nabla$-autoparallel (resp.~$\nabla^*$-autoparallel) if it is a convex set in the $\theta$ coordinates (resp.~$\eta$ coordinates). Consider a maximal $\nabla$-autoparallel submanifold $\mathcal{E}_k$ with dimension $k \leq d = \dim \mathcal{S}$. Given $P_0 \in \mathcal{E}_k$, we may  write 
\begin{equation} \label{eqn:E.k.2}
\mathcal{E}_k = \{P \in \mathcal{S}: \theta_P - \theta_{P_0} \in A_0 \},
\end{equation}
where $A_0 \subset \mathbb{R}^d$ is a vector subspace with dimension $k$. Dually, we may consider maximal $\nabla^*$-autoparallel submanifolds $\mathcal{M}_{k}$ with dimension $d - k$ (so $k$ is the codimension). 

We are now ready to state our first result. Here we focus on projections onto a $\nabla$-autoparallel submanifold. The other case is similar and is left to the reader.

\begin{theorem} \label{thm:dual.complement}
Fix $1 \leq k < d$. Consider a $\nabla$-autoparallel submanifold $\mathcal{E}_k$ and $P_0 \in \mathcal{S}$. Then there exists a unique $\mathcal{M}_k$ such that $\mathcal{E}_k \cap \mathcal{M}_k = \{P_0\}$ and the two submanifolds meet orthogonally, i.e., if $u \in T_{P_0} \mathcal{E}_k$ and $v \in T_{P_0} \mathcal{M}_k$, then $\langle u, v \rangle = 0$. We call $\mathcal{M}_k = \mathcal{M}_k(P_0)$ the dual complement of $\mathcal{E}_k$ at $P_0$. 
\end{theorem}
\begin{proof}
Let $\mathcal{E}_k$ be given by \eqref{eqn:E.k.2}. We wish to find a submanifold $\mathcal{M}_{k}$ of the form
\begin{equation} \label{eqn:Mk.candidate}
	\mathcal{M}_{k} = \{P \in \mathcal{S} : \eta_P - \eta_{P_0} \in B_0 \},
\end{equation}
where $B_0 \subset \mathbb{R}^d$ is a vector subspace of dimension $d - k$, such that $\mathcal{E}_k \cap \mathcal{M}_k = \{P_0\}$ and the two submanifolds meet orthogonally.

Consider the tangent spaces of $\mathcal{E}_k$ and $\mathcal{M}_k$ at $P_0$. Since $\mathcal{E}_k$ is given by an affine constraint in the $\theta$ coordinates, we have
\[
T_{P_0} \mathcal{E}_k = \left\{ u = \sum_i a_i \frac{\partial}{\partial \theta_i} \in T_{P_0} \mathcal{S}: (a_1, \ldots, a_d) \in A_0 \right\}.
\]
Similarly, we have
\[
T_p \mathcal{M}_k = \left\{ v = \sum_j b_j \frac{\partial}{\partial \eta_j} \in T_{P_0} S : (b_1, \ldots, b_d) \in B_0 \right\}.
\]
Given the subspace $A_0$, our first task is to choose a vector subspace $B_0$ such that $T_{P_0} \mathcal{S} = T_{P_0} \mathcal{E}_k \oplus T_{P_0} \mathcal{M}_k$ and $T_{P_0} \mathcal{E}_k \perp T_{P_0} \mathcal{M}_k$.

Let $u = \sum_i a_i \frac{\partial}{\partial \theta_i}, v = \sum_j b_j \frac{\partial}{\partial \eta_j} \in T_{P_0}\mathcal{S}$. Regard $a = (a_1, \ldots, a_d)^{\top}$ and $b = (b_1, \ldots, b_d)^{\top}$ as column vectors. Writing \eqref{eqn:metric} in matrix form, we have
\begin{equation} \label{eqn:metric.matrix.form}
	\langle u, v \rangle = \sum_{i, j} a_i b_j \left\langle \frac{\partial}{\partial \theta_i}, \frac{\partial}{\partial \eta_j} \right\rangle = a^{\top} G b,
\end{equation}
where the matrix $G = G(P_0)$ is invertible and is given by
\begin{equation} \label{eqn:metric.matrix.form2}
	G(P_0) = \frac{-1}{\Pi} I + \frac{\alpha}{\Pi^2} \theta_{P_0} \eta_{P_0}^{\top}.
\end{equation}

With this notation, we see that by letting $B_0$ be the subspace
\begin{equation} \label{eqn:right.orthogonal.complement}
	B_0 = \{b \in \mathbb{R}^d : a^{\top} G b = 0 \ \forall a \in A_0 \},
\end{equation}
and defining $\mathcal{M}_k$ by \eqref{eqn:Mk.candidate}, we have $T_{P_0} \mathcal{S} = T_{P_0} \mathcal{E}_k \oplus T_{P_0} \mathcal{M}_k$ and $T_{P_0} \mathcal{E}_k \perp T_{P_0} \mathcal{M}_k$. Regarding $(a, b) \mapsto a^{\top} G b$ as a nondegenerate bilinear form, we may regard $B_0 = A_0^{\perp}$ as the (right) orthogonal complement of $A_0$ with respect to $G$.

It remains to check that $\mathcal{E}_k \cap \mathcal{M}_k = \{P_0\}$. Suppose on the contrary that $\mathcal{E}_k \cap \mathcal{M}_k$ contains a point $P$ which is different from $P_0$. From the definition of $\mathcal{E}_k$ and $\mathcal{M}_k$, we have
\[
\theta_P - \theta_{P_0} \in A_0 \setminus \{0\}, \quad \eta_P - \eta_{P_0} \in B_0 \setminus \{0\}.
\]
Let $\gamma$ be the primal geodesic from $P_0$ to $P$ which is a straight line from $\theta_{P_0}$ to $\theta_P$ after a time change. Writing $\dot{\gamma}(0) = \sum_i a_i \frac{\partial}{\partial \theta_i} \in T_{P_0} \mathcal{E}_k$, we have $a = \lambda (\theta_P - \theta_{P_0})$ for some $\lambda > 0$. Dually, if $\gamma^*$ is the dual geodesic from $P_0$ to $P$ and $\dot{\gamma}^*(0) = \sum_j b_j \frac{\partial}{\partial \eta_j} \in T_{P_0} M_k$, then $b = \lambda' (\eta_P - \eta_{P_0})$ for some $\lambda' > 0$. It follows that
\[
\langle \dot{\gamma}(0), \dot{\gamma}^*(0) \rangle = \lambda \lambda' (\theta_P - \theta_{P_0})^{\top} G (\eta_P - \eta_{P_0}) = 0.
\]
Thus the two geodesics meet orthogonally at $P_0$. By the generalized Pythagorean theorem, we have
\[
{\bf D}[P_0 : P ] + {\bf D}[P : P_0] = {\bf D}[P_0 : P_0] = 0,
\]
which is a contradiction since $P \neq P_0$ implies ${\bf D}[P_0 : P ], {\bf D}[P_0 : P ] > 0$.	\qed
\end{proof}

\begin{remark}[Comparison with the dually flat case]
Theorem \ref{thm:dual.complement} is different from the corresponding result in \cite[Section 3B]{A01}. In \cite{A01}, given a dually flat manifold, Amari considered a $k$-cut coordinate system and showed that the $e$-flat $\mathcal{E}_k(c_k+)$ is orthogonal to the $m$-flat $\mathcal{M}_k(d_k-)$ for any values of $c_k$ and $d_k$. Since the Riemannian metric \eqref{eqn:metric} is different, the construction using $k$-cut coordinates no longer works in our setting. 
\end{remark}

As in the dually flat case, the orthogonal complement can be interpreted in terms of primal and dual projections. 

\begin{definition} [Primal and dual projections]
	Let $\mathcal{A} \subset \mathcal{S}$ and $P \in \mathcal{S}$. Consider the problem
	\begin{equation} \label{eqn:dual.projection.opt}
		{\bf D} [ \mathcal{A} : P] := \inf_{Q \in \mathcal{A}} {\bf D}[ Q : P].
	\end{equation}
	An optimal solution to \eqref{eqn:dual.projection.opt} is called a dual projection of $P$ onto $\mathcal{A}$ and is denoted by $\proj_{\mathcal{A}}^*(P)$. Similarly, the primal projection of $P$ onto $\mathcal{A}$ is defined by $\argmin_{Q \in \mathcal{A}} {\bf D}[P : Q]$ and is denoted by $\proj_{\mathcal{A}}$.
\end{definition}

The following theorem gives a geometric interpretation of the dual complement $\mathcal{M}_k$.

\begin{theorem} \label{thm:dual.complement.meaning}
	Consider a submanifold $\mathcal{E}_k$ and let $P_0 \in \mathcal{E}_k$. Let $\mathcal{M}_k = \mathcal{M}_k(P_0)$ be the dual complement of $\mathcal{E}_k$ at $P_0$ given by Theorem \ref{thm:dual.complement}. Then, for any $P \in \mathcal{M}_k$ we have $P_0 = \proj_{\mathcal{E}_k}^*(P)$. In fact, we have
	\begin{equation} \label{eqn:Mk.alternative}
		\mathcal{M}_k(P_0) = (\proj_{\mathcal{E}_k}^*)^{-1}(P_0) = \{P \in \mathcal{S}: \proj_{\mathcal{E}_k}^*(P) = P_0 \}.
	\end{equation}
	Consequently, for $P_0, P_1 \in E_k$ with $P_0 \neq P_1$, we have $M_k(P_0) \cap M_k(P_1) = \emptyset$.
\end{theorem}
\begin{proof}
	Let $P \in \mathcal{M}_k$. By definition of $\mathcal{M}_k$, the dual geodesic $\gamma^*$ from $P$ to $P_0$ is orthogonal to $\mathcal{E}_k$. Since $\mathcal{E}_k$ is $\nabla$-autoparallel, for any $Q \in \mathcal{E}_k$, $Q \neq P$, the primal geodesic $\gamma$ from $Q$ to $P_0$ lies in $\mathcal{E}_k$, and its orthogonal to $\gamma^*$. By the generalized Pythagorean theorem, we have
	\[
	{\bf D}[Q : P] =  {\bf D}[P_0 : P] + {\bf D}[ Q : P_0] > {\bf D}[P_0 : P].
	\]
	It follows that $P_0 = \proj_{\mathcal{E}_k}^*(P)$. This argument also shows that the dual projection onto $\mathcal{E}_k$, if exists, is unique.
	
	Conversely, suppose $P \in \mathcal{S} \setminus \mathcal{M}_k$. Then, by definition of $\mathcal{M}_k$, the dual geodesic from $P$ to $P_0$ is not orthogonal to $\mathcal{E}_k$. This violates the first order condition of the optimization problem \eqref{eqn:dual.projection.opt}. Hence $P_0$ is not the dual projection of $P$ onto $\mathcal{E}_k$ and we have $\mathcal{M}_k = (\proj_{\mathcal{E}_k}^*)^{-1}(P_0)$. \qed
\end{proof}

We have shown that the dual complements are disjoint and correspond to preimages of the dual projections. We complete the circle of ideas by stating the dual foliation. See Figure \ref{fig:PCA.output} for a graphical illustration in the context of principal component analysis with respect to an $L^{(\alpha)}$-divergence.

\begin{corollary} [Dual foliation] \label{cor:dual.foliation}
	Let $\mathcal{E}_k$ be given. Suppose that for each $P \in \mathcal{S}$ the infimum in \eqref{eqn:dual.projection.opt} is attained. Then
	\begin{equation} \label{eqn:dual.foliation}
		\mathcal{S} = \bigcup_{P_0 \in \mathcal{\mathcal{E}}_k} \mathcal{M}_k(P_0),
	\end{equation}
	where the union is disjoint. We call \eqref{eqn:dual.foliation} a dual foliation of $\mathcal{S}$.
\end{corollary}

\begin{remark}
In \cite{Murataetal04} the dual foliation derived from a Bregman divergence was used to study a $\mathcal{U}$-boost algorithm. It is interesting to see if the logarithmic divergence -- which satisfies a generalized Pythagorean theorem and induces a dual foliation -- leads to a class of new algorithms.
\end{remark}

\section{PCA with logarithmic divergences} \label{sec:pca}
Motivated by the success of Bregman divergence, it is natural to consider statistical applications of the $L^{(\alpha)}$-divergence. In this section we consider dimension reduction problem with the logarithmic divergences. Principal component analysis (PCA) is a fundamental technique in dimension reduction \cite{JC16}. The most basic and well-known version of PCA operates by projecting orthogonally the data onto a lower dimensional affine subspace in order to minimize the sum of squared errors. Equivalently, this can be phrased as a maximum likelihood estimation problem where each data point is normally distributed and the mean vectors are constrained to lie on a lower dimensional affine subspace. Using the duality between exponential family and Bregman divergence (which can be regarded as a generalization of the quadratic loss), Collins et al.~\cite{CDS02} formulated a nonlinear extension of PCA to exponential family. Here, we propose to replace the Bregman divergence by an $L^{(\alpha)}$-divergence, and the exponential family by the $\mathcal{F}^{(\alpha)}$-family \eqref{eqn:F.alpha.family}. We call the resulting problem $L^{(\alpha)}$-PCA.

We assume that an $\alpha$-exponentially concave function $\psi$ is given on the dual domain $\Omega$ which we also call the data space (state space of data). We define the primal parameter by $\theta = \mathsf{T}^{-1}(\eta) = \frac{\mathsf{D} \psi(\eta)}{1 - \alpha \mathsf{D} \psi(\eta) \cdot \eta}$ which takes values in the primal domain $\Theta$ (parameter space). For $k \leq d = \dim \Theta$ fixed, let $\mathsf{A}_k(\Theta)$ be the set of all $A \cap \Theta$ where $A \subset \mathbb{R}^d$ is a $k$-dimensional affine subspace.

\begin{figure}[t!]
\centering
\begin{tikzpicture}[scale = 0.6]
		
\draw (-5,0) ellipse (3 and 4);
		
\draw[blue!30, fill = blue!30] (-7.5, -1.2) -- (-2.7, -0.2) -- (-2.2, 1) -- (-6.9, 0);
\node [below, blue] at (-2.5, 0) {\footnotesize $\mathcal{A}$};
\node [below, blue] at (-5, -1) {\tiny $\nabla$-autoparallel};

\draw (5,0) ellipse (3.2 and 3.9);	
\draw[->, thick] (-1,0) -- (1, 0);
\node [above] at (0, 0) {\footnotesize $\mathsf{T}$};	
		
\node [below] at (-5, -4.2) {\footnotesize parameter space $\Theta$};
\node [below] at (5, -4.2) {\footnotesize data space $\Omega$};
		
\draw[blue!30, fill = blue!30] (2.5, -1.8)  to[out=00,in=210] (7.5, -0.4)
		to[out=80, in=-60] (7, 2.5)
		to[out = 220, in = 10] (3, 1)
		to[out = 270, in = 70] (2.5, -1.8);
		\node [below, blue] at (6.5, -1.3) {\footnotesize $\mathcal{E}= \mathsf{T}(\mathcal{A})$};

		\draw[thick, gray] (4, 3) -- (5, 0);
		\node[circle, draw=, fill = black, inner sep=0pt, minimum size=3pt, label = right: {\tiny $y$}] at (4, 3)  {};
		\node [darkgray, above] at (3.6, 1.8) {\tiny dual};
		\node [darkgray, above] at (3.6, 1.4) {\tiny geodesic};
		\node [right] at (4.6, 1) {\tiny ${\bf L}_{\psi}^{(\alpha)}[ y : \hat{\eta}]$};

		\node[circle, draw=red, fill = red, inner sep=0pt, minimum size=3pt] at (5, 0)  {};
		\node [below, red] at (5, 0) {\tiny $\hat{\eta} = \mathsf{T}(\hat{\theta})$};

	\end{tikzpicture}
	\caption{Geometry of $L^{(\alpha)}$-PCA. Here $\hat{\eta}$ is the dual projection of $y$ onto $\mathcal{E}$.}
	\label{fig:PCA.geometry}
\end{figure}
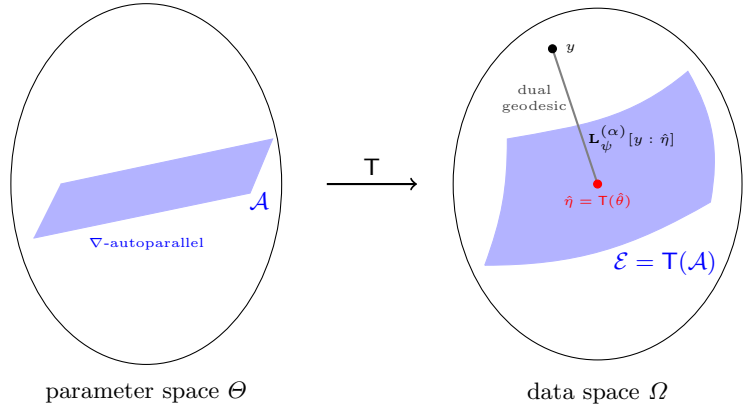

\begin{definition} [$L^{(\alpha)}$-PCA]
Let data points $y(1), \ldots, y(N) \in \Omega$ be given, and let $k \leq d$. The $L^{(\alpha)}$-PCA problem is
\begin{equation} \label{eqn:L.alpha.PCA}
\min_{\mathcal{A} \in \mathsf{A}_k(\Theta) } \min_{\theta(i) \in \mathcal{A}} \sum_{i = 1}^N {\bf L}_{\psi}^{(\alpha)} [ y(i)  : \eta(i) ], \quad \eta(i) = \mathsf{T}(\theta(i)).
\end{equation}
\end{definition}

The geometry of $L^{(\alpha)}$-PCA is illustrated in Figure \ref{fig:PCA.geometry}. A $k$-dimensional affine subspace $\mathcal{A} \subset \Theta$ is given in the parameter space $\Theta$ and provides dimension reduction of the data. The $\alpha$-Legendre transform $\mathsf{T}$ can be regarded as a ``link function'' that connects the parameter and data spaces. Through the mapping $\mathsf{T}$ we obtain a submanifold $\mathcal{E} = \mathsf{T}(\mathcal{A})$ which is typically curved in the data space but is $\nabla$-autoparallel under the induced geometry. From \eqref{eqn:self.dual}, the inner minimization in \eqref{eqn:L.alpha.PCA} is solved by letting $\eta(i) = \mathsf{T}(\theta(i))$ be the dual projection of $y(i)$ onto $\mathcal{E}$. Consequently, the straight line (dual geodesic) between $y(i)$ and $\eta(i)$ in the data space is $g$-orthogonal to $\mathcal{E}$. Finally, we vary $\mathcal{A}$ over $ \mathsf{A}_k(\Theta)$ to minimize the total divergence.

\begin{figure}[t!]
	\centering
	\includegraphics[scale=0.55]{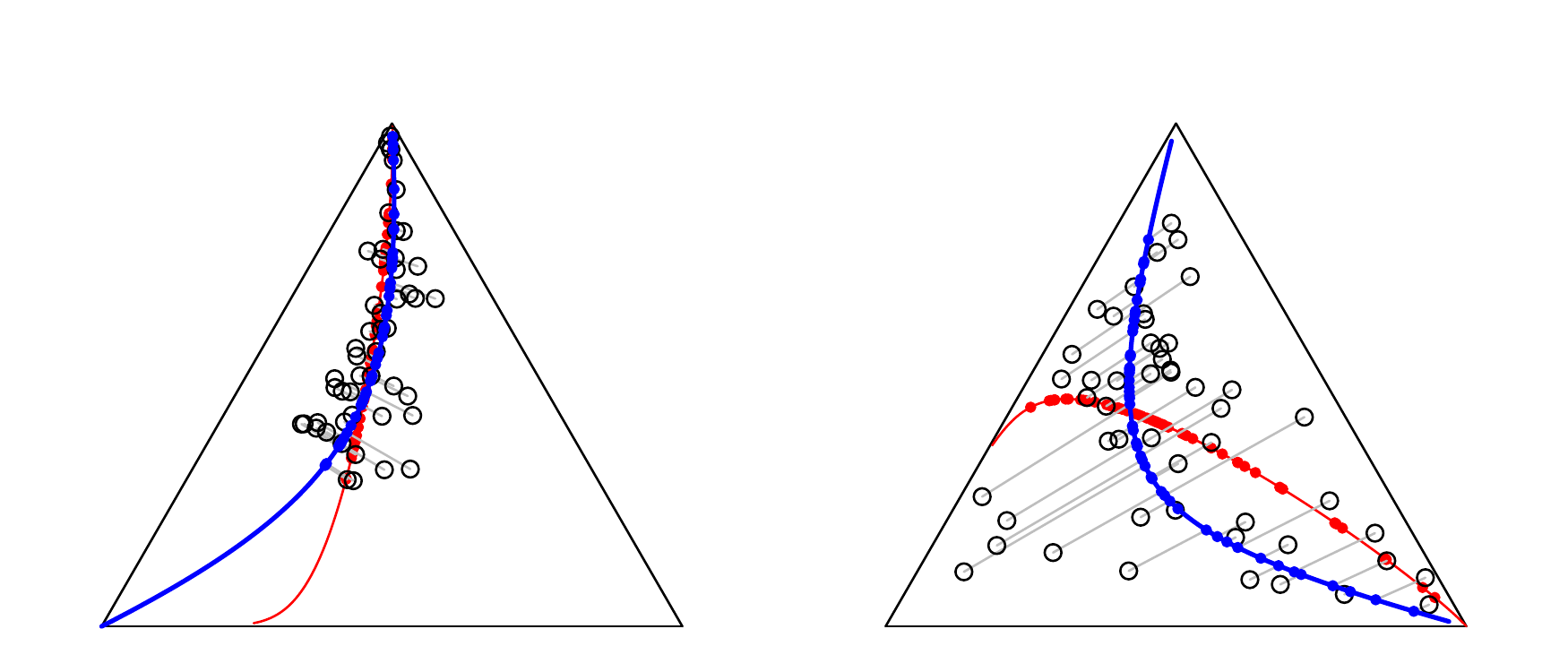}
	\vspace{-0.2cm}
	\caption{Two sample outputs of $L^{(\alpha)}$-PCA and Aitchison-PCA in the context of Example \ref{eg:Dir}. Black circles: Data points. Blue (solid) curve: Optimal $\nabla$-autoparallel submanifold $\mathcal{E} = \mathsf{T}(\mathcal{A})$ shown in the state space $\Delta_n$. Grey: Dual geodesics from data points to projected points on $\mathcal{E}$. The dual geodesics, which are straight lines in $\Delta_n$, are orthogonal to $\mathcal{E}$ by construction and form part of the dual foliation constructed in Corollary \ref{cor:dual.foliation}. Red (thin) curve: First principal component from Aitchison-PCA.}
	\label{fig:PCA.output}
\end{figure}

\begin{remark} [Probabilistic interpretation]
Consider an $\mathcal{F}^{(\alpha)}$-family \eqref{eqn:F.alpha.family}. Under suitable conditions, the density can be expressed in the form
\begin{equation} \label{eqn:F.alpha.log}
p(x; \theta) = e^{- {\bf L}_{\psi}^{(\alpha)}[ F(x) : \eta] - \psi(y)},
\end{equation}
where $\psi$ is the $\alpha$-conjugate of $\varphi$. Taking logarithm, we see that $L^{(\alpha)}$-PCA can be interpreted probabilistically in terms of maximum likelihood estimation, where each $\theta(i)$ is constrained to lie in an affine subspace of $\Theta$. In this sense the $L^{(\alpha)}$-PCA extends the exponential PCA to the $\mathcal{F}^{(\alpha)}$-family.
\end{remark}

To illustrate the methodology we give a concrete example. 

\begin{example}[Dirichlet perturbation] \label{eg:Dir}
Let $\Delta_n = \{p \in (0, 1)^n : \sum_i p_i = 1\}$ be the open unit simplex in $\mathbb{R}^n$. Consider the divergence
\begin{equation} \label{eqn:Dir.cost}
c(p, q) = \log \left( \frac{1}{n} \sum_{i = 1}^n \frac{q_i}{p_i} \right) - \frac{1}{n} \sum_{i = 1}^n \log \frac{q_i}{p_i},
\end{equation}
which is the cost function of the Dirichlet transport problem \cite{PW18b} and corresponds to the negative log-likelihood of the Dirichlet perturbation model $Q = p \oplus D$, where $p \in \Delta_n$, $D$ is a Dirichlet random vector and $\oplus$ is the Aitchison perturbation (see \cite[Section 3]{PW18b}). It is a multiplicative analogue of the additive Gaussian model $Y = x + \epsilon$. To use the framework of \eqref{eqn:L.alpha.PCA}, write $\eta_i = \frac{p_i}{p_n}$ and $y_i = \frac{q_i}{q_n}$ for $i = 1, \ldots, d := n - 1$, so that the (transformed) data space is $\Omega = (0, \infty)^d$, the positive quadrant. Then \eqref{eqn:Dir.cost} can be expressed as an $L^{(1)}$-divergence by $c(p, q) = {\bf L}_{\psi}^{(1)}[ y : \eta]$, where $\psi(y) = \frac{1}{n} \sum_{i = 1}^n \log y_i$. The $1$-Legendre transform is given by $\theta_i = \frac{1}{\eta_i} = \frac{p_n}{p_i}$, so the (transformed) parameter space is $\Theta = (0, \infty)^d$. In Figure \ref{fig:PCA.output} we perform $L^{(\alpha)}$-PCA for two sets of simulated data, where $d = 2$ (or $n = 3$) and $k = 1$. Note that the traces of dual geodesics are straight lines in the (data) simplex $\Delta_n$ \cite{PW18}. In Figure \ref{fig:PCA.output} we also show the outputs of the popular Aitchison-PCA based on the ilr-transformation \cite{EPMB03,PB11}. It is clear that the geometry of Dirichlet-PCA, which is non-Euclidean, can be quite different.
\end{example}

In future research, we plan to develop the $L^{(\alpha)}$-PCA carefully, including a rigorous treatment of \eqref{eqn:F.alpha.log}, optimization algorithms and statistical applications. 

%
%
\bibliographystyle{splncs04}
\bibliography{geometry.ref}

\end{document}